\documentclass[12pt]{amsart}
\usepackage{amsfonts}
\usepackage{amssymb}
\usepackage{amsmath}
\usepackage{mathrsfs}
\usepackage{bbm}
\usepackage{amsmath}
\usepackage{mathrsfs}
\usepackage{amsfonts}
\usepackage{amsmath,amssymb,amsfonts}
\usepackage[all]{xy}
\usepackage{hyperref}
\usepackage{a4wide}
\usepackage[mathscr]{eucal}

\theoremstyle{plain}
\newtheorem{thm}{Theorem}[section]
\newtheorem{lem}[thm]{Lemma}
\newtheorem{cor}[thm]{Corollary}
\newtheorem{prop}[thm]{Proposition}

\theoremstyle{definition}
\newtheorem*{setup}{\bf The Setup for \S 3.1}
\newtheorem*{setup2}{\bf The Setup for \S 3.2}
\newtheorem*{thma}{Remark}
\newtheorem*{con1}{Condition (1)'}
\newtheorem*{con2}{Condition (c)'}
\newtheorem{defn}[thm]{Definition}

\newtheorem{rmk}[thm]{Remark}

\numberwithin{thm}{section}

\newcommand{\alg}{{\rm alg}}

\newcommand{\Strat}{{\rm Strat}}
\newcommand{\Hom}{{\rm Hom}}
\newcommand{\Mor}{{\rm Mor}}

\newcommand{\Spec}{{\rm Spec \,}}

\newcommand{\Conn}{{\rm Conn}}

\newcommand{\Coker}{{\rm Coker}}

\newcommand{\A}{{\mathbb A}}

\newcommand{\C}{{\mathbb C}}

\newcommand{\N}{{\mathbb N}}

\newcommand{\Q}{{\mathbb Q}}

\newcommand{\T}{{\mathbb T}}

\newcommand{\Mod}{\text{\sf Mod}}

\newcommand{\Rep}{\text{\sf Rep}}

\begin{document}
\title{The Homotopy Sequence of the Algebraic Fundamental Group}
\author{ Lei Zhang }
 \address{
Universit\"at Duisburg-Essen, FB6, Mathematik, 45117 Essen, Germany}
\email{lei.zhang@uni-due.de}
\thanks{This work was supported by the Sonderforschungsbereich/Transregio 45 "Periods, moduli
spaces and the arithmetic of algebraic varieties" of the DFG}
\date{April 15, 2012}

\begin{abstract} In this paper, we prove that
the homotopy sequence of the algebraic fundamental group is exact if
the base field is of characteristic 0. So in particular, the K\"unneth
formula holds in characteristic 0.
\end{abstract}

\maketitle
\section{Introduction}

Let  $f:X\to S$ be a separable proper surjective morphism with
geometrically connected fibres between locally noetherian connected
schemes, $x\hookrightarrow X$ be a geometric point, $s=f(x)$, ${X_s}$ be the fibre at the geometric point $s$.
Grothendieck proved in \cite[Expos\'e X, Corollaire 1.4]{SGA1} that
there is a homotopy exact sequence for the \'etale fundamental group:
$$\pi_1^{\text{\'et}}({X_s},x)\to\pi_1^{\text{\'et}}(X,x)\to\pi_1^{\text{\'et}}(S,s)\to1.$$
In particular, one can take $X,Y$ to be two locally
noetherian connected $k$-schemes with $k=\bar{k}$ and assume that $Y$ is
proper separable over $k$. If we take a $k$-point $z=(x,y): \Spec(k)\to
X\times_kY$,  we can get a canonical homomorphism of group schemes
$$\pi_1^{\text{\'et}}(X\times_kY,z)\to\pi_1^{\text{\'et}}(X,x)\times\pi_1^{\text{\'et}}(Y,y).$$
Applying the homotopy exact sequence we see that the canonical homomorphism is an isomorphism. This is called
the K\"unneth formula for the \'etale fundamental group.

If $X$ is a smooth geometrically connected scheme over a field $k$
with a rational point $x\in X(k)$, we can consider the category of
$O_X$-coherent $D_{X/k}$-modules
$\Mod_c(D_{X/k})$. Let $\omega_x$ be the functor
$\Mod_c(D_{X/k})\to {\rm Vec}_k$ sending any $O_X$-coherent
$D_{X/k}$-module $M$ to $M|_{x}$ (the restriction of $M$ at $x\in X$). The category
$\Mod_c(D_{X/k})$ together with $\omega_x$ is a neutral Tannakian
category, and its Tannakian group $\pi^{\alg}(X,x)$ is defined to be
the algebraic fundamental group of $(X,x)$.  $\pi^{\alg}(X,x)$ is functorial in $(X,x)$. So if we have a
proper smooth map $f: X\to S$ between two smooth connected schemes
over a field $k$ with geometrically connected fibres, if
$x\in X(k)$ and $f(x)=s$, then we get a
sequence of maps\begin{equation}\label{homotopy sequence}
\pi^{\alg}(X_s,x)\to\pi^{\alg}(X,x)\to\pi^{\alg}(S,s)\to 1.\tag{$*$}\end{equation}

\begin{thm}Let $f: X\to S$ be a proper smooth
morphism between two smooth connected schemes of finite type over a
field $k$ of characteristic 0. Suppose that all the geometric fibres of $f$ are connected.
Let $x\in X(k)$, $s\in S(k)$ and $f(x)=s$. Then the homotopy sequence {\rm(\ref{homotopy sequence})}
 is exact.
\end{thm} In the proof, we first apply the exactness criterion \ref{General Theory} in the general Tannakian category theory to the neutral Tannakian category of  $O_X$-coherent
$D_{X/k}$-modules with a fibre functor induced by a "generic geometric point". We check that each condition of the criterion is satisfied in this situation, and deduce that the homotopy sequence is exact for schemes with a "generic geometric point". Then we use a little transcendental method to show that if the homotopy sequence is exact for schemes with a "generic geometric point" then it is exact for schemes with arbitrary $k$-rational base point. Since the category of  $O_X$-coherent
$D_{X/k}$-modules is the same as the category of $O_X$-coherent flat connections when the base field is of characteristic 0, we are able to work with flat connections instead of the more complicated $D_{X/k}$-modules. For this reason the proof is hard to be generalized to positive characteristics. If $X$ is a scheme \textit{proper} over  $\C$, then the algebraic fundamental group is the algebraic completion of the topological fundamental group, so one might be tempted to use the homotopy exact sequence of the topological fundamental group to prove the theorem. However, it turns out that the algebraic completion functor does not behave well with respect to exactness, so we have to  work with the  category of flat connections instead. This makes our proof more algebraic.

A crucial part ($\S 3.2$) of the proof of this theorem is based on \cite{E}. The main idea of that section is from the unpublished letter.

As a direct corollary of the homotopy exact sequence in characteristic 0 we have:\begin{cor}Let $T$ and $S$ be   smooth
geometrically connected schemes over a field $k$ of characteristic 0. Assume that $T$ is  proper over $k$. Then the
canonical $k$-group scheme homomorphism
$$\pi^{\alg}(T\times_kS,(t,s))\to
\pi^{\alg}(T,t)\times_k\pi^{\alg}(S,s)$$ is an isomorphism. In other words, the K\"unneth
formula holds in characteristic 0.
\end{cor}
{\bf Acknowledgments:}  I would like to express my deepest gratitude
to my advisor H\'el\`ene Esnault for proposing to me the questions
and guiding me through this topic with patience, in fact a crucial
part (\S 3.2) in the proof of the homotopy sequenece in
characteristic 0 is from her idea. I thank my co-advisor Ph\`ung
H\^{o} Hai for numerous helpful discussions.

\maketitle
\section{Preliminaries}\subsection{The general criterion} In \cite[Appendix Theorem A.1]{EPS}, H\'el\`ene Esnault, Ph\`ung H\^{o} Hai, Xiaotao
Sun formulated a necessary and sufficient condition for the exactness of a  sequence of
Tannakian groups in terms of the corresponding tensor functors.
Since the algebraic fundamental group is defined via Tannakian
duality, to prove the exactness of the homotopy sequence {\rm(\ref{homotopy sequence})} we only need to check the condition
for our specific category -- the category of $D$-modules. For the convenience of the reader we rewrite
the condition in the following theorem. See \cite[Appendix Theorem A.1]{EPS} for details.
\begin{thm}{\rm(\cite[Appendix Theorem A.1]{EPS})}\label{General Theory} Let $L\xrightarrow{q}G\xrightarrow{p}A$ be a sequence of
homomorphisms of affine group schemes over a field $k$. It induces a
sequence of functors:
$$\Rep_k(A)\xrightarrow{p^*}\Rep_k(G)\xrightarrow{q^*}\Rep_k(L),$$
where $\Rep_k(-)$ denotes the category of finite dimensional
representations of $-$ over $k$. Then we have
\begin{enumerate}\item  The group homomorphism $p:G\to A$ is surjective (faithfully flat) if and only if $p^*\Rep_k(A)$ is a full subcategory
of $\Rep_k(G)$ and closed under taking subquotients.
\item The group homomorphism $q: L\to G$ is injective (a closed immersion) if
and only if any object of $\Rep_k(L)$ is a subquotient of an object
of the form $q^*(V)$ for some $V\in \Rep_k(G)$.
\item Assume that $q$ is a closed immersion and that $p$ is
faithfully flat. Then the sequence
$L\xrightarrow{q}G\xrightarrow{p}A$ is exact if and only if the
following conditions are fulfilled:
\begin{enumerate}\item For an object $V\in \Rep_k(G)$, $q^*V\in\Rep_k(L)$ is trivial if and only if $V\cong p^*U$ for some $U\in
\Rep_k(A)$
\item Let $W_0$ be the maximal trivial subobject of $q^*V$ in $\Rep_k(L)$. Then there exists $V_0\subseteq V$ in $\Rep_k(G)$, such that
$q^*V_0\cong W_0$.
\item Any $W$ in $\Rep_k(L)$ is embeddable in
$q^*V$ for some $V\in \Rep_k(G)$.
\end{enumerate}
\end{enumerate}
\end{thm}

\begin{rmk}\label{noninj}{\bf (I).} The equivalence condition for surjectivity can also be written as:

 \begin{con1}$p:G\to A$ is surjective if and only if $p^*$ is fully faithful and $p^*(\Rep_k(A))$ is stable under taking subobjects. \end{con1}

In fact, because $p^*$ is an exact functor, if  $p^*(\Rep_k(A))$ is stable under taking subobjects then it is stable under taking subquotients.\\\\
{\bf(II)}. In (3), if $q$ is not necessarily a closed immersion, then (c) should be replaced by:

\begin{con2}For any  $W'\in\Rep_k(G)$ and any quotient
$q^*W'\twoheadrightarrow W\in
\Rep_k(L),$ there exists $V\in \Rep_k(G)$ and
an imbedding $W\hookrightarrow
q^*V.$ \end{con2}

In fact, we could decompose the homomorphism $q: L\to G$ into a composition $L\twoheadrightarrow L'\hookrightarrow G$. Let $$\Rep_k(G)\xrightarrow{a^*} \Rep_k(L')\xrightarrow{b^*} \Rep_k(L)$$ be the corresponding tensor functors. By condition (1), (2) we know that $b^*$ is fully faithful and closed under taking subquotients and that any object in $\Rep_k(L')$ is a subquotient of $a^*W'$ where $W'\in \Rep_k(G)$. Since the exactness of $L\xrightarrow{q}G\xrightarrow{p}A\to0$ is equivalent to the exactness of $0\to L'\xrightarrow{a}G\xrightarrow{p}A\to0$, we could replace (c) by the condition that for any $W'\in\Rep_k(G)$ and any subquotient $W$ of $a^*W'$, there exists $V\in \Rep_k(G)$ and
an imbedding $W\hookrightarrow
a^*V.$ Now one can check formally that this condition is the same as (c)'.

\end{rmk}

 \maketitle
\section{The homotopy exact sequence in characteristic 0}
In this section $k$ is always a field of characteristic 0. In this
case the category $\Mod_c(D_{X/k})$ is the same as the category of
vector bundles with flat connections, so in the following we will
work purely in the category of vector bundles with flat connections
and still use $\Mod_c(D_{X/k})$ to denote this category.

\subsection{The conditions (a), (b) and the surjectivity}
\begin{setup}\label{setup1} Let $f:X\to S$ be a smooth proper morphism with geometrically connected fibres between two
smooth connected schemes of finite type over $k$, $s\in
S(k)$ be a rational point, $X_s$ be the fibre, $x\in X(k)$ be a
rational point lying above $s$, then by the functoriality of the
algebraic fundamental group we get a sequence of affine group
schemes
$$\pi^{\rm alg}(X_s,x)\to\pi^{\rm alg}(X,x)\to \pi^{\rm alg}(S,s)\to1,$$ which is called the homotopy sequence. We will show that the
sequence is exact  by checking the
conditions provided in  \ref{General Theory}.\end{setup}

\begin{thm}\label{complex} The homotopy sequence $$\pi^{\rm alg}(X_s,x)\to\pi^{\rm
alg}(X,x)\to \pi^{\rm alg}(S,s)\to1$$ is a complex, and the arrow
$\pi^{\rm alg}(X,x)\to \pi^{\rm alg}(S,s)$ is surjective.
\end{thm}

\begin{proof}
Since $s\in S(k)$ is a rational point, the pull back of any object in
$\Mod_c(D_{S/k})$ is trivial in
$\Mod_c(D_{X_s/k})$, so the sequence is a complex. To show the surjectivity of
right arrow, one has to show that the functor
$f^*:\Mod_c(D_{S/k})\to\Mod_c(D_{X/k})$  is fully faithful and
stable under taking subobject.

The fact that $f^*$ is fully faithful follows readily from the
projection formula, so we only have to show that it is stable under
taking subobject. Suppose that we have an object $(E,\nabla_E)\in
\Mod_c(D_{S/k})$, and a subobject $(F,\nabla_F)\hookrightarrow
f^*(E,\nabla_E)$.

First of all, we claim that $f_*F$ is a locally free sheaf of rank equal
to that of $F$,  and the adjunction map $f^*f_*F\to F$ is an isomorphism. Moreover, the natural
map $f_*F\to E$ imbeds $f_*F$ as a subbundle of $E$ (locally split).
In fact, for any point $t\in S$, the restriction $F|_{X_t}$ of $F$ to ${X_t}$ is a free
$O_{X_t}$-module. This is because $f^*(E,\nabla_E)|_{X_t/\kappa(t)}$
is a trivial object in $\Mod_c(D_{X_t/\kappa(t)})$, and
$(F,\nabla_F)|_{X_t/\kappa(t)}\subseteq
f^*(E,\nabla_E)|_{X_t/\kappa(t)}$, so
$(F,\nabla_F)|_{X_t/\kappa(t)}$ has to be a trivial object too. This implies that
$F|_{X_t}$ is a free $O_{X_t}$-module. Thus, by \cite[ Page 48, Chapter 2, \S 5, Corollary 2]{Mum}, $f_*F$
satisfies base change in degree 0 at each point $t\in S$. Hence $f_*F$ is a vector bundle and its rank is equal to the rank of $F$ .
Since the  maps $f^*f_*F\to F$ and $f_*F\to E$ are injective after restricting  to each  fibre $X_t$ with $t\in S$, it follows from the local criterion of flatness \cite[ Page 176, \S 22, Theorem 22.5]{Mats} that they themselves are injective and their quotients are locally free (i.e. they are imbeddings of subbundles). Furthermore, since $f^*f_*F$ and $F$ have the same rank, the map $f^*f_*F\to F$ has to be an isomorphism. This finishes the
proof of the claim.

Now from the connection $\nabla_F$, we get a map:
$$f_*F\to f_*(F\otimes_{O_X}\Omega^1_{X/k})\cong f_*(f^*f_*F\otimes_{O_X}\Omega^1_{X/k})\cong f_*F\otimes_{O_S}f_*\Omega^1_{X/k}.$$
Since $f: X\to S$ is smooth, the exact sequence $$0\to
f^*\Omega^1_{S/k}\to \Omega^1_{X/k}\to \Omega^1_{X/S}\to 0$$splits locally. Hence we have an induced injection
$$f^*(E/f_*F)\otimes_{O_X}f^*\Omega^1_{S/k}\hookrightarrow
f^*(E/f_*F)\otimes_{O_X}\Omega^1_{X/k},$$ which is just
$$E/f_*F\otimes_{O_S}\Omega^1_{S/k}\hookrightarrow E/f_*F\otimes_{O_S}f_*\Omega^1_{X/k}$$ after applying $f_*$ and the projection formula. Now look at the following commutative diagram with exact rows:
$$\xymatrix{0\ar[r]&f_*F\otimes_{O_S}\Omega^1_{S/k}\ar[r]\ar[d]&E\otimes_{O_S}\Omega^1_{S/k}\ar[r]\ar[d]&E/f_*F\otimes_{O_S}\Omega^1_{S/k}\ar[d]\ar[r]&0\\
0\ar[r]&f_*F\otimes_{O_S}f_*\Omega^1_{X/k}\ar[r]&E\otimes_{O_S}f_*\Omega^1_{X/k}\ar[r]&E/f_*F\otimes_{O_S}f_*\Omega^1_{X/k}\ar[r]&0}.$$
Since $f_*F$ maps to $f_*F\otimes_{O_S}f_*\Omega^1_{X/k}$, its image
in $E/f_*F\otimes_{O_S}f_*\Omega^1_{X/k}$ is trivial. Because
$E/f_*F\otimes_{O_S}\Omega^1_{S/k}\hookrightarrow
E/f_*F\otimes_{O_S}f_*\Omega^1_{X/k}$ is injective, $f_*F\to
E\otimes_{O_S}\Omega^1_{S/k}$ factors through
$f_*F\otimes_{O_S}\Omega^1_{S/k}$. This proves that $f_*F\subseteq
E$ is equipped with a flat connection $f_*\nabla_F$ which makes
$(f_*F,f_*\nabla_F)$ a subobject of $(E,\nabla_E)$. Clearly
$f^*(f_*F,f_*\nabla_F)\cong (F,\nabla_F)$ as subobjects of
$f^*(E,\nabla_E)$. This finishes the proof.
\end{proof}

\begin{cor}
For any object
$(E,\nabla_E)\in \Mod_c(D_{X/k})$, the natural map
$$\phi:f^*H^0_{DR}(X/S, (E,\nabla_E))=f^*f_*E^{\nabla_{X/S}}\to E$$
is horizontal (i.e.  a morphism in $\Mod_c(D_{S/k})$) with respect to the Gauss-Manin connection on the
left. Moreover, $\phi$ is
injective and imbeds $f^*H^0_{DR}(X/S,(E,\nabla_E))$ as the " maximal pull back
subobject" of $(E,\nabla_E)$  in the following
sense:

If $(M,\nabla_M)\subseteq(E,\nabla_E)\in \Mod_c(D_{X/k})$ such that
$(M,\nabla_M)=f^*(N,\nabla_N)$ for some $(N,\nabla_N)\in
\Mod_c(D_{S/k})$, then the imbedding
$(M,\nabla_M)\subseteq(E,\nabla_E)$ factors through $\phi$.
\end{cor}

\begin{proof}
The fact that $\phi$ is horizontal is from the definition of the
Gauss-Manin connection. To show that it is injective one considers
the kernel $(K,\nabla_K)$ of the map. One has: $$0\to
(K,\nabla_K)\to f^*H^0_{DR}(X/S,
(E,\nabla_E))\xrightarrow{\phi}(E,\nabla_E)$$ is exact. Since the
functor $H^0_{DR}(X/S,-)$ is left exact and $H^0_{DR}(X/S,\phi)$ is
an isomorphism, we have $H^0_{DR}(X/S,(K,\nabla_K))=0$. But by the
\ref{complex}, one has an object $(K',\nabla_{K'})\in \Mod_c(D_{S/k})$ such
that $f^*(K',\nabla_{K'})=(K,\nabla_{K})$. Thus as sheaves on $S$,
one has
$0=H^0_{DR}(X/S,(K,\nabla_K))=H^0_{DR}(X/S,f^*(K',\nabla_{K'}))\cong
K'$. This shows that $K=0$. Thus $\phi$ is injective.

Suppose $(M,\nabla_M)\subseteq(E,\nabla_E)\in \Mod_c(D_{X/k})$ such
that $(M,\nabla_M)=f^*(N,\nabla_N)$ for some $(N,\nabla_N)\in
\Mod_c(D_{S/k})$, then $N\cong f_*M^{\nabla_{X/S}}\hookrightarrow
f_*E^{\nabla_{X/S}}=H^0_{DR}(X/S,(E,\nabla_E))$. This shows that
$M\hookrightarrow E$ factors through $\phi$.
\end{proof}

\begin{thm}
For any
$(E,\nabla_E)\in \Mod_c(D_{X/k})$ the restriction of the subobject
$$(F,\nabla_F):=f^*H^0_{DR}(X/S, (E,\nabla_E))\hookrightarrow(E,\nabla_E)$$
to $X_s$ gives the maximal trivial subobject of $(E,\nabla_E)|_{X_s/k}$. So
in particular, the condition (a) and (b) of \ref{General Theory} (3) are satisfied.
\end{thm}
\begin{proof}Since the maximal trivial subobject of
$(E,\nabla_E)|_{X_s/k}$ is precisely $$f^*H^0_{DR}(X_s/k,
(E,\nabla_E)|_{X_s/k})\hookrightarrow(E,\nabla_E)|_{X_s/k},$$  the theorem follows from the base change theorem for the Gauss-Manin
connection (See, for example,  \cite[Section 8]{Katz}).
\end{proof}

\subsection{The condition (c) for a generic geometric point}
Now we come to check the condition (c) in the general criterion.
Since we are not going to show the injectivity of the very left
arrow, the condition (c) in our situation reads (see \ref{noninj}):

For any  $(E,\nabla_E)\in\Mod_c(D_{X/k})$ and any quotient
$$(E,\nabla_E)|_{X_s/k}\twoheadrightarrow(F',\nabla_{F'})\in
\Mod_c(D_{X_s/k}),$$ there exists $ (F,\nabla_{F})\in \Mod_c(D_{X/k})$ and
an imbedding $$(F',\nabla_{F'})\hookrightarrow
(F,\nabla_{F})|_{X_s/k}\in \Mod_c(D_{X_s/k}).$$ Or equivalently, by taking dual, one
can say that for any $(E,\nabla_E)\in\Mod_c(D_{X/k})$
and any subobject
$(F',\nabla_{F'})\hookrightarrow(E,\nabla_E)|_{X_s/k}\in
\Mod_c(D_{X_s/k}),$ there exists $ (F,\nabla_{F})\in \Mod_c(D_{X/k})$ and
a surjection $(F,\nabla_{F})|_{X_s/k}\twoheadrightarrow
(F',\nabla_{F'})\in \Mod_c(D_{X_s/k}).$

This condition here is quite difficult to check, but since (a) and
(b) are satisfied, (c) is now equivalent to the exactness of the
homotopy sequence. We will first prove this condition in a special
case (for a "generic geometric point") then we will show that if in
this special case the condition holds then it holds  in general. Next we will place the settings for the case of a "generic
geometric point".\begin{setup2} Let $f_0:X_0\to S_0$ be a smooth morphism
between smooth geometrically connected schemes of finite type over a
field $k_0$ of characteristic 0. Let $k$ be an algebraic extension
of the function field $\kappa(S_0)$ of $S_0$, $S:=S_0\times_{k_0}k$, $X:=X_0\times_{k_0}k$,
$f:=f_0\times_{k_0}k$. Then we get a $k$-rational point $s\in S$
which corresponds to the generic point of $S_0$. This point is called the generic geometric point of $S_0$. Let $X_s$ be the
fibre of $f$ at $s\in S(k)$. Assume there is $x\in X(k)$ such that $f(x)=s$.$$\xymatrix{
  \Spec(k) \ar@/_/[ddr]_{} \ar@/^/[drr]^{=}
    \ar[dr]|-{s}                   \\
   & S \ar[d]^{} \ar[r]_{}
                      & \Spec(k) \ar[d]_{}    \\
   & S_0 \ar[r]^{}     & \Spec(k_0)               }\ \ \ \ \xymatrix{
  \Spec(k) \ar@/_/[ddr]_{} \ar@/^/[drr]^{s}
    \ar[dr]|-{x}                   \\
   & X \ar[d]^{} \ar[r]^{f}
                      & S \ar[d]_{}    \\
   & X_0 \ar[r]^{f_0}     & S_0               }
$$\end{setup2}

\begin{prop}\label{lemma 3.4}
If $(E,\nabla_E)\in \Mod_c(D_{X/k})$, $(F',\nabla_{F'})\subseteq
(E,\nabla_E)|_{X_s/k}\in \Mod_c(D_{X_s/k})$, then $\exists$ a
non-trivial Zariski open $U_0\subseteq S_0$ and an object
$(F_0,\nabla_{F_0})\in \Mod_c(D_{f_0^{-1}(U_0)/k_0})$ with a
surjection
$(F_0,\nabla_{F_0})|_{X_s/k}\twoheadrightarrow(F',\nabla_{F'}).$
\end{prop}
\begin{proof}
According to \ref{lemma 3.5} below, we have a non-trivial Zariski open
$U_0\subseteq S_0$ and a finite \'etale covering $T_0\to U_0$ with
$\kappa(T_0)\subseteq k$ such that $(E,\nabla_E)\in \Mod_c(D_{X/k})$
is defined over $\Mod_c(D_{X_0\times_{k_0}T_0/T_0})$. We may assume
$U_0=S_0$ and let $(E_0, \nabla_{E_0})\in
\Mod_c(D_{X_0\times_{k_0}T_0/T_0})$ be the object such that
$\rho^*(E_0, \nabla_{E_0})\cong (E,\nabla_E)$ where $\rho:
X=X_0\times_{k_0}k\to X_0\times_{k_0}T_0$. Let $\alpha:
T_0\hookrightarrow S_0\times_{k_0}T_0$ be the graph of $T_0\to S_0$
and $\beta: Z_{T_0}\hookrightarrow X_0\times_{k_0}T_0$ be the
pull back of the graph:
$$\xymatrix{Z_{T_0}\ar[r]^-{\beta}\ar[d]&X_0\times_{k_0}T_0\ar[d]\\T_0\ar[r]^-{\alpha}&S_0\times_{k_0}T_0}.$$
Then the pull back $\beta^*(E_0,\nabla_{E_0})\in
\Mod_c(D_{Z_{T_0}/T_0})$ is actually defined over
$\Mod_c(D_{Z_{T_0}/k_0})$. In fact, we have the following
commutative diagram:
$$\xymatrix{Z_{T_0}\ar[r]^-{\beta}\ar[d]&X_0\times_{k_0}T_0\ar[r]^-{p}\ar[d]&X_0\ar[d]\\T_0\ar@{=}[r]&T_0\ar[r]&k_0}.$$
Thus we have maps
$$\beta^*\Omega^1_{X_0\times_{k_0}T_0/T_0}\cong\beta^*p^*\Omega^1_{X_0/k_0}\to\Omega^1_{Z_{T_0}/k_0}.$$
Note that the last arrow in the above sequence is actually coming
from the following commutative
diagramme:$$\xymatrix{Z_{T_0}\ar[r]^-{p\circ\beta}\ar[d]&X_0\ar[d]\\\Spec(k_0)\ar@{=}[r]&\Spec(k_0)}.$$
This indeed extends our connection$$\nabla_{E_0}: E_0\to
E_0\otimes_{O_{X_0\times_{S_0}T_0}}\Omega^1_{X_0\times_{k_0}T_0/T_0}\cong
E_0\otimes_{O_{X_0\times_{S_0}T_0}}p^*\Omega^1_{X_0/k_0}$$ to the
connection $$\beta^*\nabla_{E_0}: \beta^*E_0\to
\beta^*E_0\otimes_{O_{Z_{T_0}}}\Omega^1_{Z_{T_0}/k_0}.$$ Let
$\lambda: Z_{T_0}\to {Z}_{S_0}\cong X_0$. Since $T_0\to S_0$
is finite \'etale, we have $\lambda_*\beta^*(E_0,\nabla_{E_0})\in
\Mod_c(D_{X_0/k_0})$, and there is a surjection
$$\lambda^*\lambda_*\beta^*(E_0,\nabla_{E_0})\twoheadrightarrow\beta^*(E_0,\nabla_{E_0}).$$
From the Cartesian diagrams
$$\xymatrix{X_s\ar[r]^-{\iota}\ar[d]&Z_{T_0}\ar[r]^-{\beta}\ar[d]&X_0\times_{k_0}T_0\ar[r]\ar[d]&X_0\ar[d]\\k\ar[r]&T_0\ar[r]&S_0\times_{k_0}T_0\ar[r]&S_0}$$
we know that if we pull back $\beta^*(E_0,\nabla_{E_0})$ along
$\iota$ then we get $(E,\nabla_E)|_{X_s/k}$. Now let
$(F'',\nabla_{F''})$ be the inverse image of $(F',\nabla_{F'})$
under the map
$$\lambda^*\lambda_*\beta^*(E_0,\nabla_{E_0})|_{X_s/k}\twoheadrightarrow\beta^*(E_0,\nabla_{E_0})|_{X_s/k}=(E,\nabla_E)|_{X_s/k}.$$
According to \ref{lemma 3.6} below, there exists a non-trivial Zariski
open $U_0\subseteq S_0$ and $(F_0,\nabla_{F_0})\in
\Mod_c(D_{f_0^{-1}(U_0)/k_0})$ with a surjection
$(F_0,\nabla_{F_0})|_{X_s/k}\twoheadrightarrow
(F'',\nabla_{F''})\twoheadrightarrow (F',\nabla_{F'})\in
\Mod_c(D_{X_s/k})$. This completes the proof.
\end{proof}

\begin{lem}\label{lemma 3.5}
{\rm(The notations and conventions in this lemma are independent)}
Let $f: X\to S$ be a smooth morphism between two integral noetherian
schemes. Let $s\in S$ be the generic point, $\kappa(s)\subseteq k$
be a separable algebraic extension of fields, $X_k$ be the generic
fibre (corresponding to $\Spec(k)\hookrightarrow S$). Then for any
object $(F,\nabla_F)\in\Mod_c(D_{X_k/k})$ with $F$ a vector bundle,
there exists a non-empty open subset $U\subseteq S$, an integral
finite \'etale covering $T\to U$ and an object
$(E,\nabla_E)\in\Mod_c(D_{X\times_ST/T})$ which satisfy (1) the
function field of $T$ is contained in $k$; (2)
$(F,\nabla_F)\cong(E,\nabla_E)_{X_k/k}$.
\end{lem}
\begin{proof}
Let $\phi: X_k\to X$ be the canonical imbedding of the generic fibre
and assume $S=\Spec(R)$. Then we get a surjection
$\phi^*\phi_*F\twoheadrightarrow F$. Since $\phi_*F$ is the union of
its coherent subsheaves, we find a coherent subsheaf $M$ of
$\phi_*F$ with a surjection $\phi^*M\twoheadrightarrow F$. Suppose
$N\subseteq \phi^*M$ is the kernel of $\phi^*M\twoheadrightarrow F$.
It is coherent because $X$ is noetherian. Then we can collect finitely
many elements $\{x_0,\cdots,x_n\}$ in $k$ which are integral over
$R$ and a non-zero element $f\in R$ such that $N$ is defined over
$R_1:=R_f[x_0,\cdots,x_n]$. Thus $F$ is defined over $R_1$. Suppose that $E_1$ is a coherent sheaf on $X\times_RR_1$ such that
$\rho_1^*E_1\cong F$, where $\rho_1: X_k=X\times_Rk\to
X\times_RR_1$. Since the problem is local for $S$, and $F$ is
locally free, we may assume $E_1$ is locally free too. Then the map
$$E_1\otimes_{O_{X\times_RR_1}}\Omega^1_{X\times_RR_1/{R_1}}\to {\rho_1}_*\rho_1^*(E_1\otimes_{O_{X\times_RR_1}}\Omega^1_{X\times_RR_1/{R_1}})$$ is
injective. Since the $k-$linear map
$$\nabla_F: F\to F\otimes_{O_{X_k}}\Omega^1_{X_k/k}$$ can be seen as
a map
$$\rho_1^*E_1\to
\rho_1^*(E_1\otimes_{O_{X\times_RR_1}}\Omega^1_{X\times_RR_1/{R_1}}),$$
we can collect finite many elements $\{y_0,\cdots,y_n\}$ in $k$
which are integral over $R$ and a non-zero element $g\in R$ such
that $\nabla_F$ is defined over $R_2=(R_1)_g[y_0,\cdots,y_n]$ as
 a flat connection. Thus we have found $T_2:=\Spec R_2$ and
$(E_2,\nabla_{E_2})\in \Mod_c(D_{X\times_ST_2/T_2})$ such that
$\rho_2^*(E_2,\nabla_{E_2})\cong (F,\nabla_F)$ (where $\rho_2:
X\times_Sk\to X\times_ST_2$) and the generic point of $T_2$ is a
finite field extension of $\kappa(s)$. Now the map $T_2\to S$ which
is finite onto its image is \'etale at the generic point of $T_2$,
thus we get a non-empty open subset $T$ of $T_2$ such that $T$ is
finite \'etale over some non-empty open $U$ of $S$. This is
precisely what we want.
\end{proof}

\begin{lem}\label{lemma 3.6} For any object $(E,\nabla_E)\in\Mod_c(D_{X_0/k_0})$ and any
imbedding $(F', \nabla_{F'})\hookrightarrow(E,\nabla_E)|_{X_s/k}\in
\Mod_c(D_{X_s/k})$ there is a non-empty open $U_0\subseteq S_0$ and
an object $(F,\nabla_F)\in \Mod_c(D_{f^{-1}(U_0)/k_0})$ which admits
a surjection $$(F,\nabla_F)|_{X_s/k}\twoheadrightarrow(F',
\nabla_{F'})\hookrightarrow(E,\nabla_E)|_{X_s/k}\in
\Mod_c(D_{X_s/k}).$$
\end{lem}

\begin{proof}
First suppose $\kappa(S_0)\subseteq k$ is a trivial extension. Let
$r:= \dim_{O_{X_s}}(F')$. According to \cite[Theorem 5.10]{EP} we
have a subobject $(M,\nabla_M)\subseteq
(E,\nabla_E)|_{X_s/k_0}\in\Mod_c(D_{X_s/k_0})$ with a surjection
$(M,\nabla_M)|_{X_s/k}\twoheadrightarrow \det(F',\nabla_{F'})$. If
we set
$(F_1,\nabla_{F_1}):=(M,\nabla_M)\otimes_{O_{X_s}}(\wedge^{r-1}(E,\nabla_E)|_{X_s/k_0})^{\vee}$,
then it is a subobject
$$(F_1,\nabla_{F_1})\subseteq
(E,\nabla_E)|_{X_s/k_0}\otimes_{O_{X_s}}(\wedge^{r-1}(E,\nabla_E)|_{X_s/k_0})^{\vee}\in
\Mod_c(D_{X_s/k_0})$$ with a surjection
$$(F_1,\nabla_{F_1})|_{X_s/k}\twoheadrightarrow(F',\nabla_{F'})\cong \det(F',\nabla_{F'})\otimes_{O_{X_s}}(\wedge^{r-1}(F',\nabla_{F'}))^{\vee}.$$
Let $u: X_s\to X_0$ be the canonical imbedding,  then we take the
inverse image of $u_*F_1$ under the canonical map
$$E\otimes_{O_X}(\wedge^{r-1}E)^{\vee}\to u_*u^*(E\otimes_{O_X}(\wedge^{r-1}E)^{\vee})$$ and denote it by $F_2$. One can check there is a non-empty open
subscheme $U_0\subseteq S_0$ so that $F_2$ is equipped with a flat
connection on $f^{-1}(U_0)/k_0$ and becomes a subobject
$$(F_2,\nabla_{F_2})\subseteq
((E,\nabla_E)\otimes_{O_{X_0}}(\wedge^{r-1}(E,\nabla_E))^{\vee})|_{f^{-1}(U_0)/k_0}\in
\Mod_c(D_{f^{-1}(U_0)/k_0})$$ which satisfies
$(F_2,\nabla_{F_2})|_{X_s/k_0}\cong(F_1,\nabla_{F_1})$. This
finishes the special case.

Now suppose $\kappa(S_0)\subseteq k$ is a non-trivial extension. It
is clear that the map
$(F',\nabla_{F'})\hookrightarrow(E,\nabla_E)|_{X_s/k}$ is defined
over $\Mod_c(D_{X_{k'}/k'})$ where $k'$ is a finite extension of
$\kappa(S_0)$ and $X_{k'}:=X_0\times_{S_0}k'$. Thus we may assume
$k/\kappa(S_0)$ is finite. Then the map $\alpha:X_s\to
X_0\times_{S_0}\kappa(S_0)$ is finite \'etale. So we get a
surjection
$$\alpha^*\alpha_*(F',\nabla_{F'})\twoheadrightarrow
(F',\nabla_{F'})\in \Mod_c(D_{X_{s}/k})$$ and an imbedding
$$\alpha_*(F',\nabla_{F'})\hookrightarrow \alpha_*((E,\nabla_{E})|_{X_s/k})\in \Mod_c(D_{X_0\times_{S_0}\kappa(S_0)/\kappa(S_0)}).$$ Thus
it is enough to show that $\alpha_*((E,\nabla_{E})|_{X_s/k})$ is
defined in $\Mod_c(D_{f^{-1}(U_0)/k_0})$ with $U_0\subseteq S_0$
non-trivial Zariski open, since then we can apply the special case
we discussed above to get a surjection on $\alpha_*(F',\nabla_{F'})$
from some object in $\Mod_c(D_{f^{-1}(U_0)/k_0})$. Since the problem
is local on $S_0$ we may assume $S_0=\Spec(R)$. Then one can find a
finite ring extension $R\subseteq R'\subseteq k$ such that $R'$ has
quotient field $k$ (e.g., the integral closure of $R$ in $k$). Again
because our problem is local on $S_0$, one may assume $R'/R$ is
finite \'etale. Let $\beta: X_0':=X_0\times_{\Spec(R)}\Spec(R')\to
X_0$, $u: X_0\times_{S_0}\kappa(S_0)\to X_0$. Then
$u^*\beta_*\beta^*(E,\nabla_{E})\cong
\alpha_*((E,\nabla_{E})|_{X_s/k})$, but
$\beta_*\beta^*(E,\nabla_{E})\in \Mod_c(D_{X_0/k_0})$. This
completes the proof.
\end{proof}

\begin{defn} Let $\Mod_c(D_{S/k},s)$ be the category whose
objects are of the form $(U,M)$, where $U$ is an open subset of $S$
containing $s$ and $M$ is a coherent sheaf on $U$ with a flat
connection $\nabla_M$ on $U/k$, whose morphisms between two objects
$(U,M)$ and $(U',M')$ are defined by
$$\Mor((U,M),(U',M')):= \Hom_{U\cap U'}((U,M)|_{U\cap
U'},(U',M')|_{U\cap U'}).$$ Let $\Mod_c(D_{X/S/k},f,s)$ be the
category whose objects are of the form $(U,M)$, where $U$ is an open
subset of $S$ containing $s$ and $M$ is a coherent sheaf on
$f^{-1}(U)$ with a flat connection $\nabla_M$ on $f^{-1}(U)/k$,
whose morphisms between two objects $(U,M)$ and $(U',M')$ are
defined by
$$\Mor((U,M),(U',M')):= \Hom_{f^{-1}(U\cap U')}((U,M)|_{f^{-1}(U\cap U')},(U',M')|_{f^{-1}(U\cap U')}).$$\end{defn}

\begin{prop}\label{extension prop}
Let $X$ be a smooth geometrically connected scheme of finite type
over a field $k$ of characteristic 0, $U\subseteq X$ be a dense open
subscheme, then for any two objects $(E,\nabla_E), (F,\nabla_F)\in
\Mod_c(D_{X/k})$ and any morphism $f_U: (E,\nabla_E)|_{U/S}\to
(F,\nabla_F)|_{U/k}\in \Mod_c(D_{U/k})$, we can uniquely extend
$f_U$ to a morphism
$$f: (E,\nabla_E)\to (F,\nabla_F)\in \Mod_c(D_{X/k}).$$
\end{prop}
\begin{proof}
The uniqueness is clear: if we have two extensions $f$ and $f'$
then the set of points on which $f=f'$ is closed in $X$, but it already contains an open dense subset, so it has to be the whole of $X$. Thus by \ref{restriction lemma} we may assume $f_U$ is an isomorphism.

Now suppose for any point $x\in X\setminus U$ we can extend $f_U$ to
a neighborhood of $x$, then using Zorn's lemma we can extend $f_U$
to a map on $X$. Hence the problem is local. We may assume
$X=\Spec(A)$ to be a smooth integral $k$-algebra with an \'etale
coordinate $X\rightarrow \A^r_k$ ($r=\dim X$), and
$E=F=A^{\oplus n}$. Let $K$ be the fraction field of $A$. The local map $f_U$ provides us with a commutative diagram:
 $$\xymatrix{
K^{\oplus n}\ar[d]_{\nabla_E\otimes_AK}\ar[r]^{f_K}&K^{\oplus n}\ar[d]^{\nabla_F\otimes_AK}\\K^{\oplus n}\otimes_{K}\Omega^1_{K/k}\ar[r]^-{f_K\otimes
id}&K^{\oplus n}\otimes_{K}\Omega^1_{K/k}}.$$
We need to show that $f_K$ takes $A^{\oplus n}\subseteq K^{\oplus n}$ to $A^{\oplus n}\subseteq K^{\oplus n}$. Or, equivalently, $f_K$ takes any vector $e_i\in K^{\oplus n}$ whose $i$-th entry is 1 and other entries are $0$ to a vector whose entries are regular for any codimension 1 point $p\in \Spec(A)$.

For this purpose, we fix a codimension 1 point $p\in \Spec(A)$, $\pi$ a uniformizer in $A_p$. Let $\hat{A}_p$ be the completion of $A_p$, $\hat{K}$ be the fraction field of $\hat{A}_p$.
It suffices to check that the base change map $f_{\hat{K}}$ of $f_K$ sends $\hat{A}_p^{\oplus n}$ to $\hat{A}_p^{\oplus n}$. By \cite[Proposition 8.9]{Katz} we know that the connections $\hat{\nabla}_E$ and $\hat{\nabla}_F$ induced from $\nabla_E$ and $\nabla_F$ via base change $A\subseteq \hat{A}_p$ are trivial connections. This implies $f_{\hat{K}}(e_i)$ is a horizontal section for the connection $\hat{\nabla}_F\otimes_{\hat{A}_p}\hat{K}=d_{\hat{K}}^{\oplus n}$. Suppose there is an entry $w_j$ of $f_{\hat{K}}(e_i)$ which has a pole, i.e. $w_j=a\pi^{-k}$ where $a\in \hat{A}_p$ is invertible and $k$ a positive integer. Then we have $$d_{\hat{K}}(w_j)=\pi^{-k}d_{\hat{K}}(a)-ka\pi^{-k-1}d_{\hat{K}}(\pi).$$ Since $d_{\hat{K}}(a)$ is a regular 1-form and $k$ is of characteristic 0, $d_{\hat{K}}(w_j)=0$ would imply $\pi^{-1}d(\pi)$ is a regular 1-form. This is absurd.
\end{proof}

\begin{lem}\label{restriction lemma} Let $X$ be a smooth $k$-scheme, $U\subseteq X$ be an open
subset, $(E,\nabla_E)\in\Mod_c(D_{X/k})$ and suppose that there is an
injection
$$(F',\nabla_F')\hookrightarrow(E,\nabla_E)|_{U/k}\in\Mod_c(D_{U/k}).$$
Then there exists a subobject
$(F,\nabla_F)\subseteq(E,\nabla_E)\in\Mod_c(D_{X/k})$ such that
$(F,\nabla_F)|_{U/k}= (F',\nabla_F')$ as subojects of
$(E,\nabla_E)|_{U/k}$
\end{lem}
\begin{proof}Let $j: U\subseteq X$ be the inclusion. We take $F$ to
be the inverse image of $j_*F'$ under the adjunction map $E\to
j_*j^*E$. Then $F$ is a coherent sheaf, and one checks easily that
$F\to E\xrightarrow{\nabla_E} E\times_{O_X}\Omega_{X/k}^1$ factors
through $F\times_{O_X}\Omega_{X/k}^1\to
E\times_{O_X}\Omega_{X/k}^1$. Hence $F$ is equipped with a
connection $\nabla_F$ and becomes a subobject of $(E,\nabla_E)$.
Clearly $(F,\nabla_F)|_{U/k}$ is equal to $(F',\nabla_{F'})$ as
subojects (since $F|_U$ is equal to $F'$).
\end{proof}

\ref{extension prop} and \ref{restriction lemma} imply immediately the
following:\begin{lem} \label{new tannakian} The category $\Mod_c(D_{S/k},S)$  {\rm(}resp.
$\Mod_c(D_{X/S/k},f,s)${\rm)} is an abelian $k$-linear rigid tensor
category equipped with an exact faithful $k$-linear tensor
functor $(M,U)\mapsto M|_s$ {\rm(}resp. $(M,U)\mapsto M|_x${\rm)}, i.e.  a neutral Tannakian category. So we have a Tannakian
group $\hat{\pi}^{\rm \text{alg}}(S,s)$ {\rm(}resp.$\hat{\pi}^{\rm
\text{alg}}(X,x)${\rm)} associated to $\Mod_c(D_{S/k},s)$
{\rm(}resp. $\Mod_c(D_{X/S/k},f,s)${\rm)}. Furthermore, the
canonical functor $\Mod_c(D_{S/k})\to \Mod_c(D_{S/k},s)$ {\rm(}resp.
$\Mod_c(D_{X/k})\to \Mod_c(D_{X/S/k},f,s)${\rm)} is fully faithful
and stable under taking subquotients. Thus we get a canonical
surjection $\hat{\pi}^{\rm
\text{alg}}(S,s)\twoheadrightarrow{\pi}^{\rm \text{alg}}(S,s)$
{\rm(}resp. $\hat{\pi}^{\rm
\text{alg}}(X,x)\twoheadrightarrow{\pi}^{\rm
\text{alg}}(X,x)${\rm)}.\end{lem}

Using results in \S  3.1 and applying  \ref{lemma 3.4}
to $\Mod_c(D_{X/S/k},f,s)$ and $\Mod_c(D_{S/k},s)$ we get:
\begin{thm}
The homotopy sequence$$\pi^{\rm \text{alg}}(X_s,x)\to\hat{\pi}^{\rm
\text{alg}}(X,x)\to\hat{\pi}^{\rm \text{alg}}(S,s)\to 1$$ is exact.
Moreover, there is a commutative diagram
$$\xymatrix{\pi^{\rm \text{alg}}(X_s,x)\ar[r]\ar@{=}[d]&\hat{\pi}^{\rm \text{alg}}(X,x)\ar[r]\ar@{>>}[d]&\hat{\pi}^{\rm \text{alg}}(S,s)\ar[r]\ar@{>>}[d]&1\\
\pi^{\rm \text{alg}}(X_s,x)\ar[r]&\pi^{\rm
\text{alg}}(X,x)\ar[r]&\pi^{\rm \text{alg}}(S,s)\ar[r]&1}$$ where all the vertical arrows are surjective.
\end{thm}

\begin{thm}\label{Theorem 3.12}
The
homotopy sequence
$$\pi^{\rm \text{alg}}(X_s,x)\to\pi^{\rm \text{alg}}(X,x)\to\pi^{\rm
\text{alg}}(S,s)\to 1$$ is exact.
\end{thm}
\begin{proof}
From the surjectivity of $\hat{\pi}^{\rm
\text{alg}}(X,x)\twoheadrightarrow \pi^{\rm \text{alg}}(X,x)$ we
know that the image of $\pi^{\rm \text{alg}}(X_s,x)$ is a normal
subgroup of $\pi^{\rm \text{alg}}(X,x)$. Taking
$G:=\Coker(\pi^{\rm \text{alg}}(X_s,x)\to\pi^{\rm
\text{alg}}(X,x))$  we get a surjective map of group schemes
$G\twoheadrightarrow \pi^{\rm \text{alg}}(S,s)$. From condition (a)
we know the functor
$$\Rep_k(\pi^{\rm
\text{alg}}(S,s))\to \Rep_k(G)$$ is essentially surjective, while
the surjectivity of $G\twoheadrightarrow \pi^{\rm \text{alg}}(S,s)$
tells us that
$$\Rep_k(\pi^{\rm
\text{alg}}(S,s))\to \Rep_k(G)$$ is an equivalence of categories.
This finishes the proof.
\end{proof}
\subsection{The general case}
In this subsection, we come back to the general settings. All the assumptions are the same as those in \S3.1.
\begin{prop}\label{base change}
Let $k\subseteq k'$ be  a field extension. Let $f' : X'\to S'$ be the base change of the
morphism $f: X\to S$ from $k$ to $k'$, $x',s'$ be the base change of $x,s$ respectively. If the sequence of
$k'$-group scheme homomorphisms$$\pi^{\rm \text{alg}}({X}_{s'}',x')\to\pi^{\rm
\text{alg}}(X',x')\to\pi^{\rm \text{alg}}(S',s')\to 1$$ is exact, then the sequence of $k$-group scheme homomorphisms
$$\pi^{\rm
\text{alg}}(X_s,x)\to\pi^{\rm \text{alg}}(X,x)\to\pi^{\rm
\text{alg}}(S,s)\to 1$$ is also exact.
\end{prop}
\begin{proof} Let $\mathfrak{C}(X')$ be the full subcategory of
$\Mod_c(D_{X'/k'})$ whose objects are those whose push forward along
the projection $X'\to X$ are the inductive limits of their coherent
subojects (i.e. subojects belonging to $\Mod_c(D_{X/k})$). This
$\mathfrak{C}(X')$ is a Tannakian subcategory and its Tannaka dual
group is precisely $\pi^{\rm \text{alg}}(X,x)\times_kk'$
\cite[10.38, 10.41]{De1}. But it is clear that this full subcategory
is also stable under taking subquotients. Thus the canonical map
$\pi^{\rm \text{alg}}(X',x')\to \pi^{\rm \text{alg}}(X,x)\times_kk'$
is surjective. The same argument applies to $X_s$ and $S$. Hence we
get a commutative diagram with the first row being exact
$$\xymatrix{\pi^{\rm
\text{alg}}(X'_{s'},x')\ar[r]^{a'}\ar@{>>}[d]&\pi^{\rm \text{alg}}(X',x')\ar[r]\ar@{>>}[d]&\pi^{\rm \text{alg}}(S',s')\ar@{>>}[d]\ar[r]&1\\
\pi^{\rm \text{alg}}(X_{s},x)\times_kk'\ar[r]^{a}&\pi^{\rm
\text{alg}}(X,x)\times_kk'\ar[r]&\pi^{\rm
\text{alg}}(S,s)\times_kk'\ar[r]&1}.$$ Since the image of $a'$ is
normal  the image of $a$ is normal too. Hence the image of
$\pi^{\rm \text{alg}}(X_{s},x)\to \pi^{\rm \text{alg}}(X,x)$ is
normal. Using the same argument employed in \ref{Theorem 3.12} we conclude
the proof of this proposition.
\end{proof}

\begin{thm}Let $f: X\to S$ be a proper smooth
morphism between two smooth connected schemes of finite type over a
field $k$ of characteristic 0. Suppose that all the geometric fibres of $f$ are connected.
Let $x\in X(k)$, $s\in S(k)$ and $f(x)=s$. Then the homotopy sequence
$$\pi^{\rm \text{alg}}(X_s,x)\to\pi^{\rm \text{alg}}(X,x)\to\pi^{\rm
\text{alg}}(S,s)\to 1$$ is exact.
\end{thm}

\begin{proof} We have already checked the conditions (a) (b) and surjectivity in \S 3.1, so the only thing left is to show condition (c). For
any object $(E,\nabla_E)\in \Mod_c(D_{X/k})$ and any morphism
$$\delta: (F',\nabla_{F'})\subseteq
(E,\nabla_E)|_{X_s/k}\in\Mod_c(D_{X_s/k}),$$ there is a finitely
generated field over $\Q$ on which all these objects ($X$, $S$,
$(E,\nabla_E)$, $\cdots$) and morphisms ($f$, $x$, $s$, $\delta$,
$\cdots$) are defined. So we can reduce our problem to the case when
$k$ is a finitely generated field over $\Q$. But in light of \ref{base change} we can assume our field $k$ is actually $\C$.

Let $K$ be the algebraic closure of the function field of $S$. Since
$K$ and $\C$ have the same transcendental degree over $\Q$, they are
isomorphic as fields. Now $\eta:\Spec(K)\hookrightarrow S$ is a
geometric generic point, so by the discussion in \S 3.2 the
sequence$$\pi^{\rm \text{alg}}({X}_{\eta},\eta')\to\pi^{\rm
\text{alg}}(X_K,\eta')\to\pi^{\rm \text{alg}}(S_K,\eta)\to 1$$ is
exact (where $X_K$ (resp. $S_K$) is the base change of $X$ (resp.
$S$) from $k$ to $K$, and $\eta'$ is any chosen $K$-rational point
of $X_K$ above $\eta$). From \ref{3.15} we get a commutative
diagram of $K$-group schemes
$$\xymatrix{\pi^{\rm \text{alg}}({X}_{\eta},\eta')\ar[r]\ar[d]^{\cong}&\pi^{\rm \text{alg}}(X_K,\eta')\ar[r]\ar[d]^{\cong}&\pi^{\rm \text{alg}}(S_K,\eta)\ar[d]^{\cong}\ar[r]&1
\\\pi^{\rm \text{alg}}({X}_{s_K},x_K)\ar[r]&\pi^{\rm \text{alg}}(X_K,x_K)\ar[r]
&\pi^{\rm \text{alg}}(S_K,s_K)\ar[r]&1}.$$ Thus the last row is
exact. Then we can conclude our theorem by \ref{base change}.
\end{proof}

\begin{lem}\label{3.15} If $f:X\to S$ is a smooth proper morphism between two smooth quasi-compact
geometrically connected $\C$-schemes with geometrically connected
fibres, $x,x'$ and $s,s'$ are $\C$-rational points of $X$ and $S$
respectively with the property that $f(x)=s$ and $f(x')=s'$, then there exists a
commutative diagram of $\C$-group schemes
$$\xymatrix{\pi^{\rm \text{alg}}(X_{s'},x')\ar[r]\ar[d]^{\cong}&\pi^{\rm \text{alg}}(X,x')\ar[r]\ar[d]^{\cong}&\pi^{\rm \text{alg}}(S,s')\ar[d]^{\cong}\\
\pi^{\rm \text{alg}}(X_{s},x)\ar[r]&\pi^{\rm
\text{alg}}(X,x)\ar[r]&\pi^{\rm \text{alg}}(S,s)}.$$
\end{lem}
\begin{proof} From the sequences of $\C$-schemes: $$X_s\to X\xrightarrow{f}
S\ \ \ \ \ \ \ \ \ \ \text{and}\ \ \ \ \ \ \ \ \ X_{s'}\to
X\xrightarrow{f} S
$$ one gets sequences of analytic spaces: $$X^{an}_s\to X^{an}\xrightarrow{f^{an}} S^{an}\ \ \ \ \ \ \ \ \ \text{and}\ \ \ \ \ \ \ \ \ \ X^{an}_{s'}\to X^{an}\xrightarrow{f^{an}} S^{an},$$
where $X^{an}_s$ (resp. $X^{an}_{s'}$) is still the fibre of $s\in
S^{an}$ (resp. $s'\in S^{an}$) under $f^{an}$, for the functor
$-^{an}$ commutes with fibre product \cite[Expos\'e XII,
1.2]{SGA1}. Now applying the first homotopy functor (in topology) to these
analytic spaces one gets a commutative diagram:
$$\xymatrix{\pi_1^{\text{top}}(X^{an}_{s'},x')\ar[r]&\pi_1^{\text{top}}(X^{an},x')\ar[r]\ar[d]^{\cong}&\pi_1^{\text{top}}(S^{an},s')\ar[d]^{\cong}\\
\pi_1^{\text{top}}(X^{an}_{s},x)\ar[r]&\pi_1^{\text{top}}(X^{an},x)\ar[r]&\pi_1^{\text{top}}(S^{an},s)}.$$
In fact by carefully choosing a path between $x$ and $x'$, there
exists a group isomorphism
$$\pi_1^{\text{top}}(X^{an}_{s},x)\xrightarrow{\cong}\pi_1^{\text{top}}(X^{an}_{s'},x')$$ making the above diagramme commutative.

To show this, one first defines a subset $Z\subseteq S^{an}$ consisting
of points $t\in S^{an}$ which satisfy that there exists a point $y\in X^{an}$ lying above $t$, a path
$\alpha$ between $x$ and $y$, and an isomorphism
$$\pi_1^{\text{top}}(X^{an}_{s},x)\xrightarrow{\cong}\pi_1^{\text{top}}(X^{an}_{t},y)$$
making the diagram
$$\xymatrix{\pi_1^{\text{top}}(X^{an}_{s},x)\ar[r]\ar[d]^{\cong}&\pi_1^{\text{top}}(X^{an},x)\ar[d]^{\alpha}\\
\pi_1^{\text{top}}(X^{an}_{t},y)\ar[r]&\pi_1^{\text{top}}(X^{an},y)}$$
commutative. $Z$ is both open and closed, since for any $t\in
S^{an}$ by Ehresmann's theorem ($f^{an}$ is proper smooth by
\cite[Expos\'e XII, proposition 3.1 et proposition 3.2]{SGA1}) there exists a neighborhood $U$ of $t\in S^{an}$ such that ${f^{an}}^{-1}(U)$
is isomorphic to $X^{an}_t\times S^{an}$ as a topological space.
Thus for any $t'\in U$ one gets a commutative diagram
$$\xymatrix{\pi_1^{\text{top}}(X^{an}_t,y)\ar[r]\ar[d]^{\cong}&\pi_1^{\text{top}}(X^{an},y)\ar[d]^{\cong}\\\pi_1^{\text{top}}(X^{an}_{t'},y')\ar[r]&\pi_1^{\text{top}}(X^{an},y')}$$
by choosing any points $y,y'\in {f^{an}}^{-1}(U)$ and any path
$y\to y'$ inside ${f^{an}}^{-1}(U)$. Hence $t\in Z$ if and only
if $t'\subseteq Z$. This shows that $Z$ is both open and closed. On
the other hand, we know that $S^{an}$ is connected
(\cite[Expos\'e XII, proposition 2.4]{SGA1}). Thus $Z=S^{an}$ and
in particular $s'\in Z$.

Now let us denote the category of integrable analytic connections on
$X^{an}$ and $X_s^{an}$, $X_{s'}^{an}$ by $\Conn(X^{an})$ and
$\Conn(X_s^{an})$, $\Conn(X_{s'}^{an})$ respectively. By
Riemann-Hilbert correspondence one has a 2-commutative diagram of
neutral Tannakian categories (the $k$-linear tensor functors in
the diagram also respect the fibres functors):
$$\xymatrix@R=0.5cm{
                &         \Conn(X_{s}^{an}) \ar[dd]^{\cong}\ar[r]^-{\cong} &\Rep_{\C}(\pi_1^{\text{top}}(X^{an}_{s},x))\ar[dd]^{\cong}    \\
  \Conn(X^{an}) \ar[ur]^{\lambda} \ar[dr]_{\lambda'} &           &     \\
                &         \Conn(X_{s'}^{an})\ar[r]^-{\cong}     &\Rep_{\C}(\pi_1^{\text{top}}(X^{an}_{s'},x'))        }.$$
Set $\iota$ to be the canonical functor $\Mod_c(D_{X/\C})\to
\Conn(X^{an})$ sending an integrable algebraic connection on $X$ to
an integrable analytic connection on $X^{an}$. Then we have a
2-commutative diagram of neutral Tannakian categories:
$$\xymatrix{&&\Mod_c(D_{X_s/\C})\ar[r]^{\cong}&\Conn(X_s^{an})\ar[dd]^{\cong}\\\Mod_c(D_{S/\C})\ar[r]&\Mod_c(D_{X/\C})\ar[r]^{\iota}\ar[ur]^{\widetilde{\lambda}}\ar[dr]_{\widetilde{\lambda'}}&\Conn(X^{an})\ar[ur]^{\lambda}\ar[dr]_{\lambda'}&
\\&&\Mod_c(D_{X_{s'}/\C})\ar[r]^{\cong}&\Conn(X_{s'}^{an})}.$$
Now applying  Tannakian duality we can conclude the proof.
\end{proof}


\begin{thebibliography}{9999999}
\bibitem[De1]{De1} P.Deligne, Le Groupe Fondamental de la Droit Projective Moins Trois Points: in Galois Group over $\Q$, Springer-Verlag,
1989.
\bibitem[De2]{De2} P.Deligne, Cat\'egories Tannakiennes: in
The Grothendieck Festschrift, Volume II, Birk\"auser, 1990.
\bibitem[E]{E} H.Esnault., A letter to Ph\`ung H\^{o} Hai (Unpublished), 26-12-2010.
\bibitem[EP]{EP} H.Esnault, P.H.Hai, The
Gauss-Manin connection and Tannaka duality, International
Mathematics Research Notices Volume 2006, Article ID 93978, Pages
1-35.
\bibitem[EPS]{EPS} H.Esnault, P.H.Hai, X.Sun, On Nori's Fundamental Group Scheme, Progress in Mathematics,
Vol.265, 366-398, Birkh\"{a}user Verlag Basel/Switzerland, 2007.
\bibitem[Gies]{Gies} D.Gieseker, Flat vector bundles and the
fundamental group in non-zero characteristic, Ann. Scu. Norm. Sup.
Pisa, 4. s\'erie, tome 2 (1975), no.1, 1-31. MR0382271 (52 \#3156).
\bibitem[Katz]{Katz} N. Katz, Nilpotent connections and the monodromy theorem:
applications of a result of Turrittin. IHES, 39, (1970) pp. 175-
232.
\bibitem[Mats]{Mats} H.Matsumura, Commutative Ring Theory, Cambrige Studies in Advanced Mathematics, 2000.
\bibitem[Mum]{Mum} D.Mumford, Abelian Varieties, Tata Institute of Fundamental Research, Hindustan Book Agency, 2008.
\bibitem[SGA1]{SGA1} A.Grothendieck, Rev\^etments \'Etales et Groupe Fondamental, Springer-Verlag,
1971.

\end{thebibliography}
\end{document}